\documentclass[a4paper,12pt]{article}
\usepackage{graphicx}
\usepackage[T1]{fontenc}
\usepackage[utf8]{inputenc}
\usepackage{amsmath,amsthm,amssymb,amsfonts}
\usepackage{lmodern}
\usepackage[margin=2.5cm]{geometry}
\usepackage{framed}
\usepackage{enumerate}
\usepackage{algorithm}
\usepackage{algorithmic}
\usepackage[colorlinks=true,urlcolor=blue,linkcolor=black,citecolor=blue]{hyperref}

\newtheorem{theorem}{Theorem}
\newtheorem{lemma}[theorem]{Lemma}

\theoremstyle{remark}
\newtheorem{remark}[theorem]{Remark}

\newcommand{\n}[1]{\left\|#1 \right\|} 

\newcommand{\la}{\lambda}

\newcommand{\ga}{\gamma}

\newcommand{\N}{\mathbb N}

\newcommand{\lr}[1]{\left\langle #1\right\rangle} 

\DeclareMathOperator{\Id}{Id}

\newcommand{\Hilbert}{\mathcal{H}}
\newcommand{\setto}{\rightrightarrows}

\title{Backward-Forward-Reflected-Backward Splitting for Three Operator Monotone Inclusions}

\author{Janosch Rieger\thanks{School of Mathematics,
                  		     Monash University,
		                     9 Rainforest Walk,
		                     Clayton VIC 3800, \textsc{Australia}.
		                     Email:~\href{mailto:janosch.rieger@monash.edu}
		                                 {janosch.rieger@monash.edu}}
	    \and
	   Matthew K. Tam\thanks{School of Mathematics \& Statistics,
	   	              	     The University of Melbourne,
	   	              	     Parkville VIC 3010, \textsc{Australia}.
	   	              	     Email:~\href{mailto:matthew.tam@unimelb.edu.au}
	   	              	                 {matthew.tam@unimelb.edu.au}}
	   	             \thanks{Institute for Numerical and Applied Mathematics,
	   	                     University of G\"ottingen,
	   	                     Lotztestr.~16--18, 37083 G\"ottingen, \textsc{Germany}}
}	   	              	                 

\begin{document}
\maketitle

\begin{abstract}
	In this work, we propose and analyse two splitting algorithms for finding a zero of the sum of three monotone operators, one of which is assumed to be Lipschitz continuous. Each iteration of these algorithms require one forward evaluation of the Lipschitz continuous operator and one resolvent evaluation of each of the other two operators. By specialising to two operator inclusions, we recover the forward-reflected-backward and the reflected-forward-backward splitting methods as particular cases. The inspiration for the proposed algorithms arises from interpretations of the aforementioned reflected splitting algorithms as discretisations of the continuous-time proximal point algorithm. 
\end{abstract}



\section{Introduction}
In this work, we propose two new splitting algorithms for finding a zero of the sum of three monotone operators in a real Hilbert space $\Hilbert$ with inner-product 
$\lr{\cdot,\cdot}$ and induced norm $\n{\cdot}$. Precisely, we consider monotone inclusions of the form
\begin{equation}
\label{eq:3op mono inc}
0 \in (A+B+C)(x),
\end{equation}
where $A,C\colon\Hilbert\setto\Hilbert$ are maximally monotone 
operators, the operator $B\colon\Hilbert\to\Hilbert$ is single-valued, monotone 
and Lipschitz continuous, and $(A+B+C)^{-1}(0)\neq\emptyset$. 
We are particularly interested in the case when $B$ is not cocoercive.  
This situation arises, for instance, when considering the first order optimality condition for saddle-point problems 
of the form
\begin{equation}\label{eq:spp}
\min_{x\in\Hilbert_1}\max_{y\in\Hilbert_2} f_1(x)+f_2(x)+\Phi(x,y) - g_1(y)-g_2(y)
\end{equation}
where $f_1,f_2\colon\Hilbert_1\to(-\infty,+\infty]$ and $g_1,g_2\colon\Hilbert_2\to(-\infty,+\infty]$ are proper
lower semicontinuous convex functions and $\Phi\colon\Hilbert_1\times\Hilbert_2\to\mathbb{R}$ is a smooth 
convex-concave function. 
Precisely, the optimality condition for \eqref{eq:spp} is given by \eqref{eq:3op mono inc} with 
$\Hilbert=\Hilbert_1\times\Hilbert_2$ and
\begin{equation}\label{eq:ABC}
  A(x,y) = \binom{\partial f_1(x)}{\partial g_1(y)}, \quad B(x,y)=\binom{\phantom{-}\nabla_x\Phi(x,y)}{-\nabla_y\Phi(x,y)}, \quad   C(x,y) = \binom{\partial f_2(x)}{\partial g_2(y)}.
\end{equation}
The operator $B$ in \eqref{eq:ABC} is Lipschitz continuous whenever $\nabla\Phi$ is,
but even in the simple realisation of \eqref{eq:spp} where $\Phi$ is a bilinear
form, the operator $B$ fails to be cocoercive \cite[Section~1]{forward2018malitsky}. 
Splitting algorithms which do not require cocoercivity of $B$ are therefore 
of general interest for solving the saddle-point problem \eqref{eq:spp}, 
and they currently attract particular attention because of their success 
in training \emph{generative adversarial networks} \cite{daskalakis2018training,ryu2019ode,mishchenko2019revisiting}. 

Until recently, most known splitting algorithms could only directly solve monotone inclusions
 with two operators, instead resorting to a higher-dimensional product space reformulation 
when more than two operators were involved (see, for instance, \cite[Proposition~26.4]{bauschke2017convex}). One of the first schemes to overcome this for three operator inclusions was proposed by Davis~\&~Yin in~\cite{davis2017three} which, in turn, generalises earlier work by  Raguet, Fadili \& Peyr\'e \cite{raguet2013generalized}.
In this connection, see \cite{giselsson2019nonlinear,johnstone2018projective}.
\emph{Davis--Yin splitting} for \eqref{eq:3op mono inc} with stepsize $\la>0$ is given by
\begin{equation}\label{eq:dys}
\left\{\begin{aligned}
 x_k &= J_{\la A}(z_k) \\
 y_k &= J_{\la C}\bigl(2x_k-z_k-\la B(x_k)\bigr) \\
 z_{k+1} &= z_k + y_k - x_k,
\end{aligned}\right.
\end{equation}
where $J_T:=(I+T)^{-1}$ denotes the \emph{resolvent} operator of a maximally monotone operator $T\colon\Hilbert\setto\Hilbert$ which is a single-valued operator with full-domain \cite[Chapter~23]{bauschke2017convex}.
When $B=0$, the method \eqref{eq:dys} reduces to \emph{Douglas--Rachford splitting} and, when $A=0$, it reduces to the \emph{forward-backward method} given by
\begin{equation}\label{eq:fb:intro}
  x_{k+1} = J_{\la C}\bigl(x_k-\la B(x_k)\bigr).
\end{equation}
Thus, like for the forward-backward method, it is necessary that $B$ is cocoercive to 
guarantee convergence of \eqref{eq:dys}. 
Consequently, \eqref{eq:dys} cannot be used to solve \eqref{eq:spp}.

In order to overcome this shortcoming, Ryu~\&~V\~u proposed the 
\emph{forward-reflected-Douglas--Rachford splitting} method in
\cite{ryu2019finding}, which, for stepsizes $\ga>\la>0$, takes the form
\begin{equation}\label{eq:frdr}
\left\{\begin{aligned}
x_{k+1} &= J_{\la A}\bigl(x_k-\la z_k-2\la B(x_k)+\la B(x_{k-1})\bigr) \\
y_{k+1} &= J_{\ga C}\bigl( 2x_{k+1}-x_k+\la z_k \bigr) \\
z_{k+1} &= z_k + \frac{1}{\la}(2x_{k+1}-x_k-y_{k+1}). 
\end{aligned}\right.
\end{equation}
The discovery of this method was computer-assisted and used ideas from the \emph{performance estimation methodology} \cite{taylor2017smooth,ryu2018operator,drori2014performance}.  When $B=0$, the method \eqref{eq:frdr} reduces to Douglas--Rachford splitting, and, when $C=0$, it reduces to the \emph{forward-reflected-backward method} \cite{forward2018malitsky} given by
  \begin{equation}
  \label{eq:forb}
  x_{k+1} = J_{\la C}\bigl(x_k-2\la B(x_k)+\la B(x_{k-1})\bigr) .
  \end{equation} 
Unlike the forward backward method, the iteration \eqref{eq:forb} does not require
$B$ to be cocoercive and can applied to \eqref{eq:spp} when $f_2=g_2=0$
 (see \cite{forward2018malitsky}). As a consequence, the method \eqref{eq:frdr}
 also does not require cocoercivity of $B$. A closely related method with the same 
property is the \emph{reflected-forward backward method}
 \cite{cevher2019reflected,malitsky2015projected} which takes the form
  \begin{equation}
  \label{eq:rfob}
  x_{k+1} = J_{\la C}\bigl(x_k-\la B(2x_k-x_{k-1})\bigr) .
  \end{equation} 

To the best of the authors' knowledge, an explanation for why \eqref{eq:forb} or \eqref{eq:rfob} converge without cocoercivity has not yet been given (beyond their proofs), although there are satisfactory interpretations in the case when $C=0$ (see \cite{ryu2019ode,csetnek2019shadow}).

\bigskip

In the first part of this work, we address this issue by providing interpretations of \eqref{eq:forb} and \eqref{eq:rfob} in the general case as two different discretisations of the same asymptotically stable dynamical system. This interpretation explains their convergence in the absence of cocoercivity. In the second part of this work, we use this interpretation to derive two new three operator splitting algorithms for solving \eqref{eq:3op mono inc} which, in contrast to \eqref{eq:frdr}, use the same stepsize in resolvents of $A$ and $C$. In further contrast to approach used in \eqref{eq:frdr}, our algorithms were not discovered using computer-assistance, but rather they are derived systematically through discretising a continuous-time dynamical system. 

\bigskip

The remainder of this work is structured as follows. In Section~\ref{s:interpreting}, we propose an interpretation of two reflected splitting algorithms for two operator inclusions, namely the forward-reflected-backward and the reflected-forward-backward splitting methods, as discretisations of the continuous-time proximal point algorithm. In Sections~\ref{s:BFoRB} and \ref{s:BRFoB}, we exploit an analogous technique to derive two new algorithms for solving the three operator inclusion \eqref{eq:3op mono inc} and analyse their convergence. More precisely, the scheme in Section~\ref{s:BFoRB} generalises the forward-reflected-backward method, and the scheme in Section~\ref{s:BRFoB} generalises the reflected-forward-backward splitting method.

\section{Interpreting Reflected Splitting}\label{s:interpreting}
In this section, we provide interpretations of the forward-reflected-backward \eqref{eq:forb} and the reflected forward-backward  \eqref{eq:rfob} splitting schemes as discretisations of continuous-time dynamical systems associated with the proximal point algorithm. To this end, in this section, we restrict our attention to the inclusion
  \begin{equation}
  \label{eq:mono inc}
  0\in(B+C)(x)\subseteq\Hilbert, 
  \end{equation}
where $B:\Hilbert\to\Hilbert$ and $C\colon\Hilbert\setto\Hilbert$ are maximally monotone and $B$ is $L$-Lipschitz continuous. 
In other words, we consider the three operator inclusion \eqref{eq:3op mono inc} with $A=0$.

In the absence of stronger assumptions such as cocoercivity of $B$ or strong monotonicity of $B+C$, 
the \emph{forward-backward method} need not converge \cite{chen1997convergence}. 
Recall that, with constant stepsize $\la>0$, this method takes the form
  \begin{equation}\label{eq:fb}
  x_{k+1} = J_{\la C}\bigl(x_k-\la B(x_k)\bigr),
  \end{equation}
and can be interpreted as a discretisation of the dynamical system (see \cite{abbas2014newton,attouch2018convergence})
  $$ \dot{x}(t)+x(t) = J_{\la C}\bigl(x(t)-\la B(x(t))\bigr). $$ 
As a monotone inclusion, this system takes the form
  \begin{equation}\label{eq:fb dyn}
  -\dot{x}(t)\in \la C\bigl(\dot{x}(t)+x(t)\bigr)+\la B\bigl(x(t)\bigr). 
  \end{equation}
We posit that the reliance of \eqref{eq:fb} on cocoercivity of $B$ for convergence arises from the choice of the argument of $B$ in \eqref{eq:fb dyn}. 
Indeed, by augmenting the argument of $B$, we obtain the dynamical system
  $$ -\dot{x}(t)\in \la C\bigl(\dot{x}(t)+x(t)\bigr)+\la B\bigl(\dot{x}(t)+x(t)\bigr) $$ 
which is asymptotically stable (\emph{i.e.,} it has the property that all its trajectories converge to solutions 
of the inclusion \eqref{eq:3op mono inc}) whenever the sum $B+C$ is maximally monotone. 
This fact can be shown by noting its equivalence to the continuous-time proximal point algorithm given by
  \begin{equation}
  \label{eq:ppa}
    \dot{x}(t)+x(t) = J_{\la (B+C)}\bigl( x(t) \bigr).
  \end{equation} 
  
We will now use the dynamical system \eqref{eq:ppa} to interpret the forward-reflected-backward and the reflected forward-backward splitting schemes
\eqref{eq:forb} and \eqref{eq:rfob}. 
To this end, we first decouple $B$ and $C$ to obtain
  \begin{equation}
  \label{eq:ppa alt}
  \dot{x}(t)+x(t) = J_{\la C}\bigl( x(t) -\la B\bigl(\dot{x}(t)+x(t)\bigr)\bigr). 
  \end{equation} 
In this form, the system is not explicit due to the appearance of $\dot{x}(t)$ on the right-hand side. To deal with this difficulty, we approximate $B\bigl(\dot{x}(t)+x(t)\bigr)$ using the linearisation of $B$ at $x(t)$. Denoting $y(t)=B(x(t))$ and assuming sufficient smoothness, we obtain
  \begin{equation}
  \label{eq:lin}
  B\bigl( x(t)+\dot{x}(t) \bigr) \approx B(x(t)) + \mathbf{J}_B(x(t))\dot{x}(t) = y(t)+\dot{y}(t),
  \end{equation} 
where $ \mathbf{J}_B$ denotes the \emph{Jacobian} of $B$ and the the identity $\mathbf{J}_B(x(t))\dot{x}(t)=\dot{y}(t)$ is a consequence of the chain rule. 
Substituting \eqref{eq:lin} into \eqref{eq:ppa alt} gives the system
  \begin{equation}\label{eq:coupled sys}
  \left\{\begin{aligned}
   \dot{x}(t)+x(t) &= J_{\la C}\bigl( x(t) -\la y(t)-\la \dot{y}(t)\bigr) \\
   y(t) &= B\bigl(x(t)\bigr).
  \end{aligned}\right.
  \end{equation}
Let $h>0$. We now approximate the trajectories of \eqref{eq:coupled sys} 
at the time points $(kh)_{k\in\mathbb{N}}$ by discrete trajectories
$(x_k)_{k\in\mathbb{N}}$ and $(y_k)_{k\in\mathbb{N}}$ with
$x_k\approx x(kh)$ and $y_k\approx y(kh)=B(x_{k})$. 
In this notation, using the forward discretisation $\dot{x}(t)\approx \frac{x_{k+1}-x_k}{h}$ and the backward discretisation $\dot{y}(t)\approx \frac{y_k-y_{k-1}}{h}$ in \eqref{eq:coupled sys} gives the scheme
  \begin{equation}\label{eq:forb2}
  x_{k+1} = (1-h)x_k +  hJ_{\la C}\left( x_k -\la B(x_k)-\frac{\la}{h}\bigl( B(x_k)- B(x_{k-1})\bigr) \right).
  \end{equation}
For $h\in(0,1]$, the scheme \eqref{eq:forb2} is precisely the relaxed variant 
of the forward-reflected-backward method studied in 
\cite[Section~4]{forward2018malitsky}. 
In particular, when $h=1$, \eqref{eq:forb2} recovers the standard 
forward-reflected-backward method \eqref{eq:forb}. 
Thus, in summary, \eqref{eq:forb} can be interpreted as 
a discretisation of a linearisation of the proximal point algorithm \eqref{eq:ppa}.

Alternatively, using the forward discretisation $\dot{x}(t)\approx \frac{x_{k+1}-x_k}{h}$ on the left-hand side and the backward discretisation  
$\dot{x}(t)\approx \frac{x_k-x_{k-1}}{h}$ on the right-hand side 
of equation \eqref{eq:ppa alt} leads to the scheme
  \begin{equation}\label{eq:rfb}
  x_{k+1} = (1-h)x_k+hJ_{\la C}\left( x_k -\la B\bigl(x_k+\frac{1}{h}(x_k-x_{k-1})\bigr) \right).
  \end{equation}
When $h=1$, the iteration \eqref{eq:rfb} is precisely the \emph{reflected-forward
backward method} \eqref{eq:rfob}, so the method \eqref{eq:rfob} can interpreted 
as a discretisation of the proximal point algorithm \eqref{eq:ppa}.

The fact that both schemes \eqref{eq:forb2} and \eqref{eq:rfb} converge without
cocoercivity of $B$, while the method \eqref{eq:fb} does not, can therefore be
partly explained by a connection to the proximal point algorithm, which also does
not require cocoercivity for convergence.

\section{Backward-Forward-Reflected-Backward Splitting}\label{s:BFoRB}
In this section, we use the idea of discretising a linearisation of a dynamical system 
to derive a new algorithm for solving an operator monotone inclusion. Recall that we consider the three operator monotone inclusion 
\begin{equation}
 \label{eq:3op mono inc 2}
  0 \in (A+B+C)(x),
\end{equation}
where $A,C\colon\Hilbert\setto\Hilbert$ are maximally monotone operators, 
and $B\colon\Hilbert\to\Hilbert$ is single-valued, monotone and $L$-Lipschitz
continuous. 
Note that, in this case, the sum $B+C$ is also maximally monotone \cite[Corollary~25.2]{bauschke2017convex}.

Let $\la >0$. The \emph{Douglas--Rachford splitting method} can only be directly applied to problems of finding a zero of the sum of two monotone operators. 
Applying a continuous-time version of this method to the inclusion
\eqref{eq:3op mono inc 2} yields the equation
\begin{equation}\label{eq:dr}
\left\{\begin{aligned}
x(t) &= J_{\la A}(z(t)) \\
y(t) &= J_{\la(B+C)}\bigl(2x(t)-z(t)\bigr) \\
\dot{z}(t) &= y(t)-x(t).
\end{aligned}\right.
\end{equation}
All trajectories of this dynamical system converge weakly to solutions of
the inclusion \eqref{eq:3op mono inc 2} (see \cite{boct2017dynamical,csetnek2019shadow}). 
Our aim is to derive a new algorithm in which $B$ and $C$ are decoupled. 
Proceeding as in the previous section, we rewrite the second equation in \eqref{eq:dr} as
$$ y(t) = J_{\la C}\bigl( 2x(t)-z(t)-\la B(y(t)) \bigr). $$
Let $h>0$.  Assuming sufficient smoothness, we may approximate $B(y(t))$ by using the linearisation of $B(y(\cdot))$ at $t-h$. That is,
  \begin{equation}\label{eq:By2}
  B(y(t)) \approx B(y(t-h)) + \mathbf{J}_{B\circ\hspace{0.05ex}y}(t-h)h.
  \end{equation} 
Substituting this approximation into \eqref{eq:dr} gives the system  
\begin{equation}\label{eq:dr-lin}
\left\{\begin{aligned}
x(t) &= J_{\la A}(z(t)) \\
y(t) &= J_{\la C}\bigl(2x(t)-z(t) - \la B(y(t-h)) -\la \mathbf{J}_{B\circ\hspace{0.05ex}y}(t-h)h\bigr) \\
\dot{z}(t)&= y(t)-x(t).
\end{aligned}\right.
\end{equation}
Now we approximate the trajectories of \eqref{eq:dr-lin} at the time points 
$(kh)_{k\in\mathbb{N}}$ by discrete trajectories
$(x_k)_{k\in\mathbb{N}}$, $(y_k)_{k\in\mathbb{N}}$ and $(z_k)_{k\in\mathbb{N}}$ 
with $x_k\approx x(kh)$, $y_k\approx y(kh)$ and $z_k\approx z(kh)$. 
Using the forward discretisation $\dot{z}(t)\approx\frac{z_{k+1}-z_k}{h}$ and the backward discretisation 
\[\mathbf{J}_{B\circ\hspace{0.05ex}y}(t-h)\approx \frac{B(y_{k-1})-B(y_{k-2})}{h}\] 
yields the iteration 
\begin{equation*}
\left\{\begin{aligned}
x_k &= J_{\la A}(z_k) \\ 
y_k &= J_{\la C}\bigl( 2x_k-z_k-2\la B(y_{k-1})+\la B(y_{k-2}) \bigr) \\
z_{k+1} &= z_k+h(y_k-x_k).
\end{aligned}\right.
\end{equation*}
For simplicity, we fix $h=1$, so the discrete-time system becomes
\begin{equation}\label{eq:alg}
\left\{\begin{aligned}
	x_k &= J_{\la A}(z_k) \\ 
	y_k &= J_{\la C}\bigl( 2x_k-z_k-2\la B(y_{k-1})+\la B(y_{k-2}) \bigr) \\
	z_{k+1} &= z_k+y_k-x_k.
\end{aligned}\right.
\end{equation}
We refer to \eqref{eq:alg} as the 
\emph{backward-forward-reflected-backward method} and assume throughout this
section that the initial points $z_0,y_{-1},y_{-2}\in\Hilbert$ are chosen
arbitrarily.

When $B=0$, the iteration \eqref{eq:alg} reduces to the Douglas--Rachford
method for $A+C$, and when $A=0$, it reduces to the forward-reflected-backward
splitting method for $B+C$. Note also that \eqref{eq:alg} can equivalently written as 
the system of inclusions
\begin{equation}\label{eq:inclusion}
\left\{\begin{aligned}
  \la A(x_k) &\ni z_k-x_k &\\
 \la C(y_k)  &\ni 2x_k-z_k-y_k-2\la B(y_{k-1})+\la B(y_{k-2})  \\
  z_{k+1}    &= z_k+y_k-x_k. 
\end{aligned}\right.
\end{equation}
The following lemma will be key in proving convergence of the iteration \eqref{eq:alg}.

\begin{lemma}\label{lem:key}
Consider points $x,z\in\Hilbert$ such that $z-x\in\la A(x)$ and $x-z\in\la(B+C)(x)$. 
Then the sequences 
$(x_k)_{k\in\mathbb{N}}$, $(y_k)_{k\in\mathbb{N}}$ and $(z_k)_{k\in\mathbb{N}}$ 
given by iteration \eqref{eq:alg} satisfy
\begin{multline}\label{eq:lemma}
\n{z_{k+1}-z}^2 + 2\la\lr{B(y_{k})-B(y_{k-1}),x-y_{k}} + \n{z_{k+1}-z_k}^2 \\
\leq \n{z_k-z}^2 + 2\la\lr{B(y_{k-1})-B(y_{k-2}),x-y_{k-1}} \\
+ 2\la\lr{B(y_{k-1})-B(y_{k-2}),y_{k-1}-y_{k}}.
\end{multline}	
\end{lemma}
\begin{proof}
By monotonicity of $\la A$, we have
\begin{equation}\label{eq:l1}
0 \leq \lr{(z-x)-(z_k-x_k),x-x_k}.
\end{equation}
Using monotonicity of $\la C$ and \eqref{eq:alg} gives
\begin{equation}
\label{eq:l2}
\begin{aligned}
0 &\leq \lr{(x-z)-\la B(x)+z_k-2x_k+y_k+2\la B(y_{k-1})-\la B(y_{k-2}),x-y_k} \\
  &= \lr{(x-z)-(x_k-z_k),x-x_k} + \lr{z_{k+1}-z_k,z-z_{k+1}}   \\
  &\hspace{3.2cm}+ \la \lr{ B(y_{k-1})-B(x),x-y_k} + \la \lr{B(y_{k-1})-B(y_{k-2}),x-y_k}.
\end{aligned}
\end{equation}
From monotonicity of $\la B$, it follows that
\begin{equation}
\label{eq:l3}
\la \lr{B(y_{k-1})-B(x),x-y_{k}} \leq -\la \lr{B(y_k)-B(y_{k-1}),x-y_{k}}. 
\end{equation} 
By summing the inequalities \eqref{eq:l1}, \eqref{eq:l2} and \eqref{eq:l3}, and using the identity
\begin{equation}\label{eq:parallelogram}
 \lr{z_{k+1}-z_k,z-z_{k+1}} = \frac{1}{2}\left(\n{z_k-z}^2-\n{z_{k+1}-z_k}^2-\n{z_{k+1}-z}^2  \right), 
\end{equation}
we obtain
\begin{multline*}
0 \leq \n{z_k-z}^2 - \n{z_{k+1}-z}^2 - \n{z_{k+1}-z_k}^2 \\
- 2 \la \lr{ B(y_{k})-B(y_{k-1}),x-y_k} + 2\la \lr{B(y_{k-1})-B(y_{k-2}),x-y_k},
\end{multline*}
from which the claimed inequality \eqref{eq:lemma} follows. 
\end{proof}

For the convenience of the reader and clarify of presentation, we recall the following well-known result for reference in the proof of Theorem~\ref{th:main}.

\begin{lemma}[Opial's lemma]\label{Opial}
Let $(z_k)_{k\in\N}$ be a sequence in $\Hilbert$ and let $\Omega$ be a non-empty subset of $\Hilbert$. Suppose the following assertions hold.
\begin{enumerate}[(a)]
\item\label{l:opial_a} For every $z\in\Omega$, the sequence $(\|z_k-z\|)_{k\in\N}$ converges.
\item\label{l:opial_b} Every weak sequential cluster point of $(z_k)_{k\in\mathbb{N}}$ belongs to $\Omega$. 
\end{enumerate}
Then $(z_k)_{k\in\mathbb{N}}$ converges weakly to a point in~$\Omega$.
\end{lemma}
\begin{proof}
See, for instance, \cite[Lemma~2.47]{bauschke2017convex}.
\end{proof}

The following theorem is our main result regarding convergence of the backward-forward-reflected-backward method \eqref{eq:alg}. 
In what follows, we make use of the fact that the resolvent $J_{\la A}$ is
\emph{firmly nonexpansive} (see, for instance, \cite[Proposition~23.10]{bauschke2017convex}), that is,
$$ \n{J_{\la A}(x)-J_{\la A}(y)}^2 + \n{(\Id-J_{\la A})(x)-(\Id-J_{\la A})(y)}^2 \leq \n{x-y}^2 \quad\forall x,y\in\Hilbert. $$

\begin{theorem}\label{th:main}
  Suppose $(A+B+C)^{-1}(0)\neq\emptyset$, let $\la \in(0,\frac{1}{8L})$, and
  consider sequences $(z_k)_{k\in\N}, (x_k)_{k\in\N}$ and $(y_k)_{k\in\N}$ 
  given by \eqref{eq:alg} for arbitrary initial points 
  $z_0,y_{-1},y_{-2}\in\Hilbert$. 
  Then the following assertions hold:
  \begin{enumerate}[(a)]
  	\item\label{th:main:a} The sequence $(z_k)_{k\in\N}$ converges weakly 
  	to a point $\bar{z}\in\Hilbert$.
  	\item\label{th:main:b} The sequences $(x_k)_{k\in\N}$ and $(y_k)_{k\in\N}$ converge weakly 
  	to a point $\bar{x}\in\Hilbert$.
  	\item\label{th:main:c} We have $\bar{x}=J_{\la A}(\bar{z})\in(A+B+C)^{-1}(0)$. 
  \end{enumerate}
\end{theorem}
\begin{proof}
We shall apply Lemma~\ref{Opial} to sequence $(z_k)_{k\in\N}$ and the set $\Omega$ defined by
\begin{equation} \label{eq:omega}
\Omega:=\{z\in\Hilbert: J_{\lambda A}(z)\in(A+B+C)^{-1}(0),\
(J_{\lambda A}-\Id)(z)\in\lambda(B+C)(J_{\lambda A}(z))\},
\end{equation}
which we claim is nonempty. Indeed, since $(A+B+C)^{-1}(0)\neq\emptyset$ by assumption, there exist $x\in(A+B+C)^{-1}(0)$ and $z\in\Hilbert$ such that $z-x\in\la A(x)$ and $x-z\in\la (B+C)(x)$. Combining these two inclusions yields  $x=J_{\la A}(z)$ and 
 $$(J_{\la A}-\Id)(z) = x-z \in \la (B+C)(x)=\la (B+C)(J_{\la A}(z)). $$
Thus $z\in\Omega$ and hence $\Omega\neq\emptyset$.

Next, consider an arbitrary $z\in\Omega$ and denote $x:=J_{\lambda A}(z)$. 
By Lemma~\ref{lem:key}, we have
\begin{multline}\label{eq:proof tele}
\n{z_{k+1}-z}^2 + 2\la\lr{B(y_{k})-B(y_{k-1}),x-y_{k}} + \n{z_{k+1}-z_k}^2 \\
\leq \n{z_k-z}^2 + 2\la\lr{B(y_{k-1})-B(y_{k-2}),x-y_{k-1}} \\
+ 2\la\lr{B(y_{k-1})-B(y_{k-2}),y_{k-1}-y_{k}}
\end{multline}
for all $k\in\N$.
To estimate the last term, first note that firm nonexpansivity of $J_{\la A}$ gives
\begin{equation}\label{eq:bad ineq}
\begin{aligned}
\n{y_{k-1}-y_{k-2}}^2
	  &= \n{(z_k-z_{k-1}+x_{k-1})-(z_{k-1}-z_{k-2}+x_{k-2})}^2 \\
	  &\leq 2\n{z_k-z_{k-1}}^2 + 2\n{(x_{k-1}-z_{k-1})-(x_{k-2}-z_{k-2})}^2 \\
	  &\leq 2\n{z_k-z_{k-1}}^2 + 2\n{z_{k-1}-z_{k-2}}^2,
\end{aligned}
\end{equation}
and thus Lipschitz continuity of $B$ yields
\begin{equation}\label{jrinsert}
\begin{aligned}
&\frac{2}{L}\lr{B(y_{k-1})-B(y_{k-2}),y_{k-1}-y_{k}}\\
&\quad \leq \n{y_{k-1}-y_{k-2}}^2 + \n{y_k-y_{k-1}}^2 \\
&\quad \leq 2\n{z_{k-1}-z_{k-2}}^2 + 4\n{z_k-z_{k-1}}^2 + 2\n{z_{k+1}-z_{k}}^2
\end{aligned}\end{equation}
for all $k\in\N$.
Consider the sequence $(\varphi_k)_{k\in\mathbb{N}}$ given by
\begin{multline*}
\varphi_k := \n{z_k-z}^2 + 2\la\lr{B(y_{k-1})-B(y_{k-2}),x-y_{k-1}} \\
+ \frac{3}{4}\n{z_k-z_{k-1}}^2 + 2\la L\n{z_{k-1}-z_{k-2}}^2. 
\end{multline*}
Substituting inequality \eqref{eq:proof tele} into \eqref{jrinsert}, 
and setting $\epsilon:=(1-2\la L)-3/4>0$ yields
\begin{equation}
\label{eq:key}
\varphi_{k+1} + \epsilon \n{z_{k+1}-z_k}^2 \leq \varphi_k \implies \varphi_{k+1} +\epsilon\sum_{i=0}^k\n{z_{i+1}-z_i}^2 \leq \varphi_0\quad\forall k\in\mathbb{N}.
\end{equation}
On the other hand, Lipschitz continuity of $B$, nonexpansivity of $J_{\la A}$ 
and \eqref{eq:bad ineq} implies
\begin{align*}
& \frac{2}{L}\lr{B(y_{k})-B(y_{k-1}),x-y_{k}} \\
&\quad\leq \n{y_k-y_{k-1}}^2 + \n{(z_{k+1}-x_{k+1})-(z_k-x_k)+(x_{k+1}-x)}^2 \\
&\quad\leq  2\n{z_k-z_{k-1}}^2 + 2\n{z_{k+1}-z_k}^2 + 2\n{(z_{k+1}-x_{k+1})-(z_k-x_k)}^2 +2\n{x_{k+1}-x}^2 \\
&\quad\leq  2\n{z_k-z_{k-1}}^2 + 4\n{z_{k+1}-z_k}^2 + 2\n{z_{k+1}-z}^2,
\end{align*}
which yields the lower bound
  $$ \varphi_{k+1} \geq (1-2\la L)\n{z_{k+1}-z}^2 + \left(\frac{3}{4}-4\la L\right)\n{z_{k+1}-z_k}^2 \geq \frac{3}{4}\n{z_{k+1}-z}^2\geq 0\quad\forall\,k\in\N.$$
By combining this with inequality \eqref{eq:key}, we deduce that 
$(\varphi_k)_{k\in\N}$ converges,  ${\n{z_{k+1}-z_k}\to0}$, and $(z_k)_{k\in\N}$ is bounded. Since $x_k=J_{\la A}(z_k)$ and $x=J_{\la A}(z)$, nonexpansivity of $J_{\la A}$ 
implies that $\n{x_{k+1}-x_k}\to0$ 
and $(x_k)_{k\in\N}$ is bounded. Because of the identity $y_k=z_{k+1}-z_k+x_k$, we then have that
$\n{y_{k+1}-y_k}\to0$ and $(y_k)_{k\in\N}$ is bounded. 
It then follows that
  \begin{equation*}
   \lim_{k\to\infty}\n{z_k-z}^2=\lim_{k\to\infty}\varphi_k,
  \end{equation*}
which establishes Lemma~\ref{Opial}\eqref{l:opial_a}.

Now, let $z\in\Hilbert$ be a weak sequential cluster point of $(z_k)_{k\in\N}$. Since $(x_k)_{k\in\N}$ is bounded, it follows that there exists $x\in\Hilbert$ such that $(z,x)$ is a weak sequential cluster point of $((z_k,x_k))_{k\in\N}$. Next, we note that  \eqref{eq:inclusion} implies the inclusion
\begin{multline}\label{eq:prod incl}
\binom{z_k-z_{k+1}}{z_{k}-z_{k+1}}-\la\binom{0}{B(y_{k-1})-B(y_{k-2})} 
-\la\binom{0}{B(y_{k-1})-B(y_{k})}\\
\in \left(\begin{bmatrix}(\la A)^{-1}&0 \\ 0&\la (B+C) \\ \end{bmatrix} + \begin{bmatrix} 0 & -\Id \\ \Id & 0 \\ \end{bmatrix} \right)\binom{z_k-x_k}{z_{k+1}-z_k+x_k}.
\end{multline}
Since $A$ and $B+C$ are maximally monotone operators, appealing to \cite[Lemma~1]{csetnek2019shadow} and taking the limit in \eqref{eq:prod incl} along a subsequence of $((z_k,x_k))_{k\in\N}$ which converges weakly to $(z,x)$, yields
\begin{equation*}
\binom{0}{0}
\in \left(\begin{bmatrix}(\la A)^{-1}&0 \\ 0&\la (B+C) \\ \end{bmatrix} + \begin{bmatrix} 0 & -\Id \\ \Id & 0 \\ \end{bmatrix} \right)\binom{z-x}{x}.
\end{equation*}
The first inclusion gives $z-x\in\la A(x)$, which is equivalent to $x=J_{\la A}(z)$. 
The second inclusion gives $x-z\in\la(B+C)(x)$ which implies $
(J_{\lambda A}-\Id)(z)\in\lambda(B+C)(J_{\lambda A}(z))$. Thus, altogether, we have that $z\in\Omega$ which establishes Lemma~\ref{Opial}\eqref{l:opial_b}.

Having verified all of its assumptions, we now invoke Lemma~\ref{Opial} to deduce that $(z_k)_{k\in\N}$ converges weakly to a point $\bar{z}\in\Omega$. Finally, let $\bar{x}\in\Hilbert$ be an arbitrary weak sequential cluster point of the bounded sequence $(x_k)_{k\in\N}$. Then, by using an argument analogous to the above, we deduce that $\bar{x}=J_{\la A}(\bar{z})$. Thus, $(x_k)_{k\in\N}$ possesses precisely one weak sequential cluster point and hence $(x_k)_{k\in\N}$ is weakly convergent. The remainder of the proof easily follows from the identity $y_k=z_k-z_{k+1}+x_k$ and the definition of $\Omega$. 
\end{proof}

\begin{remark}
By setting $A=0$ in \eqref{eq:alg}, we have $x_k=z_k$ and Theorem~\ref{th:main} reduces to \cite[Corollary~2.6]{forward2018malitsky}, albeit with a worse stepsize. On the other hand, setting $C=0$ gives a scheme which seems to be new, even in the two operator case.
\end{remark}

\begin{remark}\label{r:ryu}
In \cite{ryu2019finding}, Ryu~\&~V\~u proposed a related method known as the
\emph{forward-reflected-Douglas--Rachford (FRDR) splitting} for inclusion 
\eqref{eq:3op mono inc} which takes the form
\begin{equation*}
\left\{\begin{aligned}
x_{k+1} &= J_{\la A}\bigl(x_k-\la u_k-2\la B(x_k)+\la B(x_{k-1})\bigr) \\
y_{k+1} &= J_{\ga C}\bigl( 2x_{k+1}-x_k+\la u_k \bigr) \\
u_{k+1} &= u_k + \frac{1}{\la}(2x_{k+1}-x_k-y_{k+1}). 
\end{aligned}\right.
\end{equation*}
For its convergence, the constants $\la,\ga>0$ used in the resolvents of $A$ and $C$ are required to satisfy
  \begin{equation}
  \label{eq:rv constants}
    0 < \la < \frac{\ga}{1+2L\ga}.
  \end{equation} 
Note that, in particular, this means that $\la < \ga $ whereas in the setting of Theorem~\ref{th:main} both resolvents use the same constant. 
On the other hand, by taking $\ga$ sufficiently large in
\eqref{eq:rv constants}, the constant $\la$ be choose arbitrarily close 
to $\frac{1}{2L}$, in line with \cite[Corollary~2.6]{forward2018malitsky}. 
We also remark that, in practice, there is usually no advantage to having both
resolvents with the same stepsize ({\em i.e.,} $\la=\ga$).
\end{remark}

\section{Backward-Reflected-Forward-Backward Splitting}\label{s:BRFoB}
In this section, we derive a second algorithm for solving the inclusion 
\eqref{eq:3op mono inc 2} by means of discretising the system \eqref{eq:dr} which,
in decoupled form, is given by
\begin{equation}\label{eq:dr2}
\left\{\begin{aligned}
x(t) &= J_{\la A}(z(t)) \\
y(t) &= J_{\la C}\bigl( 2x(t)-z(t)-\la B(y(t)) \bigr) \\
\dot{z}(t) &= y(t)-x(t).
\end{aligned}\right.
\end{equation}
Let $h>0$. Assuming sufficient smoothness, we approximate $y(t)$ on the right-hand side of the second equation by using the linearisation of $y$ at $t-h$. That is, we have
\begin{equation}\label{eq:y2}
y(t) \approx y(t-h) + \dot{y}(t-h)h
\end{equation} 
Substituting this approximation into \eqref{eq:dr2} gives 
\begin{equation}\label{eq:dr2-lin}
\left\{\begin{aligned}
x(t) &= J_{\la A}(z(t)) \\
y(t) &= J_{\la C}\left( 2x(t)-z(t)-\la B\bigl(y(t-h) + \dot{y}(t-h)h \bigr)\right) \\
\dot{z}(t) &= y(t)-x(t).
\end{aligned}\right.
\end{equation}
Now we discretise the trajectories in \eqref{eq:dr2-lin} at the time points 
$(kh)_{k\in\mathbb{N}}$ which we denoted by $z_k:=z(kh), x_k:=x(kh)$ and 
$y_k:=y(kh)$. Using the forward discretisation $\dot{z}(t)\approx\frac{z_{k+1}-z_k}{h}$
 and the backward discretisation $\dot{y}(t-h)\approx \frac{y_{k-1}-y_{k-2}}{h}$ yields
\begin{equation*}
\left\{\begin{aligned}
x_k &= J_{\la A}(z_k) \\ 
y_k &= J_{\la C}\bigl( 2x_k-z_k-\la B(2y_{k-1}-y_{k-2} ) \bigr) \\
z_{k+1} &= z_k+h(y_k-x_k),
\end{aligned}\right.
\end{equation*}
As in Section~\ref{s:BFoRB}, we assume $h=1$ for simplicity, so that the system becomes
\begin{equation}\label{eq:alg2}
\left\{\begin{aligned}
x_k &= J_{\la A}(z_k) \\ 
y_k &= J_{\la C}\bigl( 2x_k-z_k-\la B(2y_{k-1}-y_{k-2} ) \bigr) \\
z_{k+1} &= z_k+y_k-x_k.
\end{aligned}\right.
\end{equation}

We refer to \eqref{eq:alg2} as the 
\emph{backward-reflected-forward-backward method} and assume that the initial
points $z_0,y_{-1},y_{-2}\in\Hilbert$ are chosen arbitrarily.
Note that, when $B=0$, \eqref{eq:alg2} reduces to the 
Douglas--Rachford algorithm for $A+C$, and when $A=0$, it reduces to
the reflected-forward-backward splitting method for $B+C$. 
For convenience, we use the notation
$$ \bar{z}_k=2z_k-z_{k-1},\quad \bar{x}_k=2x_k-x_{k-1},\quad 
\bar{y}_k=2y_k-y_{k-1}. $$
The iteration \eqref{eq:alg2} can then be equivalently written as the system of inclusions
\begin{equation}\label{eq:alg2 inc}
\left\{\begin{aligned}
\la A(x_k) &= z_k-x_k \\
\la C(y_k) &\ni 2x_k-z_k-y_k-\la B(\bar{y}_{k-1}) \\
z_{k+1} &= z_k+y_k-x_k.
\end{aligned}\right.
\end{equation}
Before proving weak convergence of \eqref{eq:alg2}, we require the
following preparatory lemma.
\begin{lemma}\label{l:key2}
Consider points $x,z\in\Hilbert$ such that $z-x\in\la A(x)$ and $x-z\in\la(B+C)(x)$.
Then the sequences $(z_k)_{k\in\N}, (x_k)_{k\in\N}$ 
and $(y_k)_{k\in\N}$ given by \eqref{eq:alg2} satisfy	
\begin{multline}
\n{z_{k+1}-z}^2 + 2\la\lr{B(\bar{y}_{k-1})-B(x),y_k-y_{k-1}} + 2\n{z_{k+1}-z_k}^2 
+\n{z_{k+1}-\bar{z}_k}^2 \\
\leq \n{z_k-z}^2 + 2\la\lr{B(\bar{y}_{k-2})-B(x),y_{k-1}-y_{k-2}}
+\n{z_k-z_{k-1}}^2 \\
+2\la\lr{B(\bar{y}_{k-1})-B(\bar{y}_{k-2}),\bar{y}_{k-1}-y_k}.
\end{multline}
\end{lemma}
\begin{proof}
	By monotonicity of $\la A$, we have
	\begin{equation}\label{eq:l2-1}
	0 \leq \lr{(z-x)-(z_k-x_k),x-x_k}.
	\end{equation}
	Using monotonicity of $\la C$ and iteration \eqref{eq:alg2} gives
	\begin{equation}\label{eq:l2-2}
	\begin{aligned}
	0 &\leq \lr{(x-z)-\la B(x)-2x_k+z_k+y_k+\la B(\bar{y}_{k-1}),x-y_k} \\
	&= \lr{(x-z)-(x_k-z_k),x-x_k} + \lr{z_{k+1}-z_k,z-z_{k+1}}+\la\lr{B(\bar{y}_{k-1})-B(x),x-y_k}.
	\end{aligned}  
	\end{equation}
	From monotonicity of $\la B$, it follows that
	\begin{equation}\label{eq:l2-3}
	\la \lr{B(\bar{y}_{k-1})-B(x),x-y_k} \leq \la \lr{B(\bar{y}_{k-1})-B(x),\bar{y}_{k-1}-y_k}.
	\end{equation}
	By summing the inequalities \eqref{eq:l2-1}, \eqref{eq:l2-2} and \eqref{eq:l2-3}, and using the identity \eqref{eq:parallelogram}
	we obtain
	\begin{multline}\label{eq:inter}
	\n{z_{k+1}-z}^2 + \n{z_{k+1}-z_k}^2 - \n{z_k-z}^2 \\
	\leq 2\la\lr{B(\bar{y}_{k-1})-B(\bar{y}_{k-2}),\bar{y}_{k-1}-y_k} + 2\la\lr{B(\bar{y}_{k-2})-B(x),\bar{y}_{k-1}-y_k}.
	\end{multline}
	Now, by monotonicity of $\la A$, we have
	\begin{equation}\label{eq:l2-4}
	0  \leq \lr{(z_k-x_k)-(z_{k-1}-x_{k-1}),x_k-x_{k-1}}.
	\end{equation}
	Using monotonicity of $\la C$ gives
	\begin{equation}\label{eq:l2-5}
	\begin{aligned}
	0 &\leq \lr{\left(2x_k-z_k-y_k-\la B(\bar{y}_{k-1})\right)-\left(2x_{k-1}-z_{k-1}-y_{k-1}-\la B(\bar{y}_{k-2})\right),y_k-y_{k-1}} \\
	&=  \lr{\left(x_k-z_{k+1}-\la B(\bar{y}_{k-1})\right)-\left(x_{k-1}-z_{k}-\la B(\bar{y}_{k-2})\right),y_k-y_{k-1}} \\
	&= \lr{(x_k-z_{k})-(x_{k-1}-z_{k-1}),x_k-x_{k-1}} + \lr{z_{k}-z_{k+1},(z_{k+1}-z_k)-(z_k-z_{k-1})} \\
	&\hspace{2.25cm}  + \la\lr{B(x)-B(\bar{y}_{k-1}),y_k-y_{k-1}}  + \la\lr{B(\bar{y}_{k-2})-B(x),y_k-\bar{y}_{k-1}}\\
	&\hspace{8.1cm} + \la\lr{B(\bar{y}_{k-2})-B(x),y_{k-1}-y_{k-2}}.
	\end{aligned}
	\end{equation}
	By summing \eqref{eq:l2-4} and \eqref{eq:l2-5}, and using the identity
	$$ \lr{z_{k}-z_{k+1},(z_{k+1}-z_k)-(z_k-z_{k-1})} = \frac{1}{2}\left( \n{z_k-z_{k-1}}^2-\n{z_{k+1}-z_k}^2 - \n{z_{k+1}-\bar{z}_k}^2\right), $$
    we obtain
    \begin{multline}\label{eq:inter2}
    \n{z_{k+1}-z_k}^2 + \n{z_{k+1}-\bar{z}_k}^2 - \n{z_k-z_{k-1}}^2 + 2\la\lr{B(\bar{y}_{k-1})-B(x),y_k-y_{k-1}}  \\
    \leq  2\la\lr{B(\bar{y}_{k-2})-B(x),y_{k-1}-y_{k-2}} + 2\la\lr{B(\bar{y}_{k-2})-B(x),y_k-\bar{y}_{k-1}} .
    \end{multline}
    The claimed inequality follows by summing \eqref{eq:inter} and \eqref{eq:inter2}.
\end{proof}

The following theorem is our main result regarding convergence of the 
backward-reflected-forward-backward method \eqref{eq:alg2}.
\begin{theorem}\label{th:main2}
	Suppose $(A+B+C)^{-1}(0)\neq\emptyset$, let $\la \in(0,\frac{1}{22 L})$, 
	and consider the sequences $(z_k)_{k\in\N}, (x_k)_{k\in\N}$ and $(y_k)_{k\in\N}$ 
    given by \eqref{eq:alg2} for arbitrary initial points $z_0,y_{-1},y_{-2}\in\Hilbert$.
	Then the following assertions hold:
	\begin{enumerate}[(a)]
		\item The sequence $(z_k)_{k\in\N}$ converges weakly to a point 
		$\bar{z}\in\Hilbert$.
		\item The sequences $(x_k)_{k\in\N}$ and $(y_k)_{k\in\N}$ converge weakly 
		to a point $\bar{x}\in\Hilbert$.
		\item We have $\bar{x}=J_{\la A}(\bar{z})\in(A+B+C)^{-1}(0)$. 
	\end{enumerate}
\end{theorem}
\begin{proof}
    The proof strategy is analogous to Theorem~\ref{th:main} and uses the same nonempty set $\Omega$ defined in \eqref{eq:omega}. Consider an arbitrary $z\in\Omega$ and denote $x:=J_{\lambda A}(z)$. By Lemma~\ref{l:key2}, we have
	\begin{multline}\label{eq:tele3}
	\n{z_{k+1}-z}^2 + 2\la\lr{B(\bar{y}_{k-1})-B(x),y_k-y_{k-1}} + 2\n{z_{k+1}-z_k}^2 + \n{z_{k+1}-\bar{z}_k}^2  \\
	\leq \n{z_k-z}^2 + 2\la\lr{B(\bar{y}_{k-2})-B(x),y_{k-1}-y_{k-2}} +  \n{z_k-z_{k-1}}^2 \\
       + 2\la\lr{B(\bar{y}_{k-1})-B(\bar{y}_{k-2}),\bar{y}_{k-1}-y_k}.
	\end{multline}
	We now estimate the last term in \eqref{eq:tele3}. To this end, first observe that  firm nonexpansivity of $J_{\la A}$ implies
	\begin{align*}
	&\n{(\bar{z}_{k-1}-\bar{x}_{k-1})-(z_k-x_k)}^2 \\
      &\quad\leq 2\n{(z_k-x_k)-(z_{k-1}-x_{k-1})}^2 + 2\n{(z_{k-1}-x_{k-1})-(z_{k-2}-x_{k-2})}^2 \\
	  &\quad\leq 2\n{z_{k}-z_{k-1}}^2 + 2\n{z_{k-1}-z_{k-2}}^2.
	\end{align*}
	Using this inequality, we deduce that
	\begin{equation}\label{eq:alpha ineq}
	\begin{aligned}
	  \n{\bar{y}_{k-1}-y_k}^2 
	    &= \n{(\bar{z}_k-\bar{z}_{k-1}+\bar{x}_{k-1})-(z_{k+1}-z_k+x_k)}^2 \\
	    &\leq (1+6)\n{z_{k+1}-\bar{z}_k}^2 + \left(1+\frac{1}{6}\right)\n{(\bar{z}_{k-1}-\bar{x}_{k-1})-(z_k-x_k)}^2 \\
	    &\leq 7\n{z_{k+1}-\bar{z}_k}^2 + \frac{7}{3}\left( \n{z_{k}-z_{k-1}}^2 + \n{z_{k-1}-z_{k-2}}^2\right).
	\end{aligned}
	\end{equation}
Note that the inequality \eqref{eq:bad ineq} from the proof Theorem~\ref{th:main} is still valid for \eqref{eq:alg2} as the third lines of \eqref{eq:alg} and \eqref{eq:alg2} are identical. By combining \eqref{eq:alpha ineq} with \eqref{eq:bad ineq}, we obtain
	\begin{equation*}
	\begin{aligned}
	&\n{\bar{y}_{k-1}-\bar{y}_{k-2}}^2  \\
	&\quad \leq 2\n{y_{k-1}-y_{k-2}}^2 + 2\n{y_{k-1}-\bar{y}_{k-2}}^2 \\
	&\quad \leq 4\n{z_k-z_{k-1}}^2 + \frac{26}{3}\n{z_{k-1}-z_{k-2}}^2 + \frac{14}{3}\n{z_{k-2}-z_{k-3}}^2 + 14\n{z_{k}-\bar{z}_{k-1}}^2.
	\end{aligned}
	\end{equation*}				
    Thus using \eqref{eq:alpha ineq} and the previous inequality yields
    \begin{equation}\label{eq:last estimate}
	\begin{aligned}
	&\frac{2}{L}\lr{B(\bar{y}_{k-1})-B(\bar{y}_{k-2}),\bar{y}_{k-1}-y_k} \\
	&\quad\leq \n{\bar{y}_{k-1}-\bar{y}_{k-2}}^2+\n{\bar{y}_{k-1}-y_k}^2 \\
	&\quad\leq \frac{19}{3}\n{z_k-z_{k-1}}^2 + 11\n{z_{k-1}-z_{k-2}}^2 + \frac{14}{3}\n{z_{k-2}-z_{k-3}}^2 \\
	&\hspace{5.8cm} + 7\n{z_{k+1}-\bar{z}_k}^2 + 14\n{z_{k}-\bar{z}_{k-1}}^2
	\end{aligned}		    
	\end{equation}		
We define the sequence $(\varphi_k)_{k\in\N}$ by
	\begin{multline}
	\varphi_k := \n{z_k-z}^2 + 2\la\lr{B(\bar{y}_{k-2})-B(x),y_{k-1}-y_{k-2}} +  \left(1+22\la L\right)\n{z_k-z_{k-1}}^2 \\
	  +\frac{47}{3}\la L\n{z_{k-1}-z_{k-2}}^2 
	    + \frac{14}{3}\la L\n{z_{k-2}-z_{k-3}}^2 + \frac{7}{11}\n{z_{k}-\bar{z}_{k-1}}^2.
	\end{multline}	
	Substituting \eqref{eq:last estimate} into the estimate 
	\eqref{eq:tele3} and setting $\epsilon:=1-22\la L>0$ yields
	\begin{equation}\label{eq:tele 3}
	\varphi_{k+1}+\epsilon\n{z_{k+1}-z_k}^2 \leq \varphi_k \implies 
	\varphi_{k+1}+\sum_{i=0}^k\n{z_{i+1}-z_i}^2\leq \varphi_0\quad\forall\,k\in\N.
	\end{equation}
	Next, we derive a lower bound for $\varphi_{k+1}$. To this end, note that firm nonexpansivity of $J_{\la A}$ implies
	\begin{align*}
	&\n{(\bar{z}_{k}-\bar{z}_{k-1}+\bar{x}_{k-1})-(z-z+x)}^2 \\
	&\quad \leq 2\n{\bar{z}_{k}-z}^2 + 2\n{(\bar{x}_{k-1}-\bar{z}_{k-1})-(x-z)}^2 \\
	&\quad \leq 2\n{\bar{z}_{k}-z}^2 + 4\n{(x_{k-1}-z_{k-1})-(x-z)}^2+ 4\n{(x_{k-1}-z_{k-1})-(x_{k-2}-z_{k-2})}^2 \\
	&\quad \leq 4\n{z_{k+1}-z}^2 + 4\n{z_{k+1}-\bar{z}_{k}}^2 + 4\n{z_{k-1}-z}^2 + 4\n{z_{k-1}-z_{k-2}}^2 \\
	&\quad \leq 12\n{z_{k+1}-z}^2 + 16\n{z_{k+1}-z_{k}}^2 + 16\n{z_{k}-z_{k-1}}^2  + 4\n{z_{k-1}-z_{k-2}}^2 + 4\n{z_{k+1}-\bar{z}_{k}}^2.
	\end{align*}
	Thus, using Lipschitz continuity of $B$ and  \eqref{eq:bad ineq} 
	gives
	\begin{align*}
	&\frac{2}{L}\lr{B(\bar{y}_{k-1})-B(x),y_{k}-y_{k-1}} \\
	&\quad\leq \frac{1}{2}\n{(\bar{z}_{k}-\bar{z}_{k-1}+\bar{x}_{k-1})-(z-z+x)}^2 + 2\n{y_{k}-y_{k-1}}^2 \\
    &\quad \leq 6\n{z_{k+1}-z}^2 + 12\n{z_{k+1}-z_{k}}^2 + 12\n{z_{k}-z_{k-1}}^2  + 2\n{z_{k-1}-z_{k-2}}^2 + 2\n{z_{k+1}-\bar{z}_{k}}^2.
	\end{align*}		
Altogether, we have the lower bound
	\begin{multline*}
	\varphi_{k+1} \geq (1-6\la L) \n{z_{k+1}-z}^2 + (1+10\la L)\n{z_{k+1}-z_{k}}^2 + \frac{11}{3}\la L\n{z_{k}-z_{k-1}}^2 \\
	+\frac{8}{3}\la L\n{z_{k-1}-z_{k-2}}^2 +\left(\frac{7}{11}-2\la L\right)\n{z_{k+1}-\bar{z}_{k}}^2 
	\geq  \frac{6}{11} \n{z_{k+1}-z}^2 \geq 0. 
	\end{multline*}
	By combining this with \eqref{eq:tele 3}, 
	we deduce that $(\varphi_k)_{k\in\N}$ converges,
	$\n{z_{k+1}-z_k}\to0$, and $(z_k)_{k\in\N}$ is bounded. 
	Arguing as in Theorem~\ref{th:main}, it then follows that 
	  $$ \lim_{k\to\infty}\n{z_k-z}^2 = \lim_{k\to\infty}\varphi_k, $$
	which establishes Lemma~\ref{Opial}\eqref{l:opial_a}. Next, we note that  \eqref{eq:alg2 inc} implies the inclusion
    \begin{multline*}
   \binom{z_k-z_{k+1}}{z_{k}-z_{k+1}} + \la\binom{0}{B(y_{k})-B(\bar{y}_{k-1})}\\
    \in \left(\begin{bmatrix}(\la A)^{-1}&0 \\ 0&\la (B+C) \\ \end{bmatrix} + \begin{bmatrix} 0 & -\Id \\ \Id & 0 \\ \end{bmatrix} \right)\binom{z_k-x_k}{z_{k+1}-z_k+x_k}.
    \end{multline*}
	The remainder of the proof is now analogous to Theorem~\ref{th:main}.
\end{proof}

\begin{remark}
  By using slightly tighter estimates in the current proof, numerics suggest that the upper bound $\la L < \frac{1}{22}$ can be improved slightly to approximately $\frac{1}{20}$. However, since the arithmetic in the resulting proof becomes significantly more complex, we have decided to present a the present slightly sub-optimal version for the sake of presentation. Moreover, as the main novelty of Theorem~\ref{th:main2} is its connection to a continuous-time dynamical system, the precise value of this upper-bound is not our main concern.
\end{remark}

\section{Conclusions}
In this work, we provided an intuitive interpretation to explain convergence of
the forward-reflected-backward and reflected-forward-backward methods in the absence of cocoercivity. More precisely, we showed that these methods can be understood as two different discretions of the continuous-time proximal point algorithm which corresponds to an asymptotically stable dynamical system. This insight allowed us to derive two new three operator splitting algorithms, neither of which relies on cocoercivity for convergence. Future work will investigate whether the insights gained from Section~\ref{s:interpreting} can be combine with the three operator resolvent-splitting scheme with minimal lifting from from \cite[Section~4]{ryu2018operator} to derive a four operator scheme which exploits forward evaluations or with \cite{artacho2019computing} to compute the resolvent of three operator sums.

\paragraph{Acknowledgements.} 
This work was supported in part by a Robert Bartnik Visiting Fellowship from the 
School of Mathematics at Monash University. MKT is the recipient of a Discovery Early Career Research Award (DE200100063) from the Australian Research Council.

\bibliographystyle{abbrv}
\bibliography{biblio}

\end{document}